\pdfoutput=1
\documentclass{article}
\usepackage[utf8]{inputenc}
\usepackage{amsmath}
\usepackage{amssymb}
\usepackage{amsthm, thmtools}
\usepackage{mathtools}
\usepackage[a4paper,margin=2.5cm]{geometry}
\usepackage{authblk}
\usepackage[colorlinks=false]{hyperref}
\usepackage{nameref, cleveref}
\usepackage{enumitem}

\hypersetup{
    pdftitle={Separating polynomial χ-boundedness from χ-boundedness},
    pdfauthor={Marcin Briański, James Davies, Bartosz Walczak}
}

\declaretheorem[name=Theorem, refname={Theorem,Theorems}, Refname={Theorem, Theorems}]{theorem}
\declaretheorem[name=Lemma, refname={Lemma, Lemmas}, Refname={Lemma, Lemmas}, sibling=theorem]{lemma}

\declaretheorem[name=Conjecture, refname={Conjecture,Conjectures}, Refname={Conjecture, Conjectures}, sibling=theorem]{conjecture}

\setlist[enumerate]{label=\textup{(\arabic*)}, noitemsep, topsep=3pt plus 3pt, leftmargin=*}

\newcommand{\calC}{\mathcal{C}}
\newcommand{\calG}{\mathcal{G}}
\newcommand{\calH}{\mathcal{H}}
\newcommand{\setN}{\mathbb{N}}
\newcommand{\setZ}{\mathbb{Z}}

\DeclarePairedDelimiter\set{\{}{\}}
\DeclarePairedDelimiter\size{\lvert}{\rvert}

\DeclarePairedDelimiter\floor{\lfloor}{\rfloor}

\let\leq\leqslant
\let\geq\geqslant

\title{Separating polynomial $\chi$-boundedness from $\chi$-boundedness}
\author[1]{Marcin Briański \thanks{marcin.brianski@doctoral.uj.edu.pl; partially supported by the Polish National Science Centre grant (BEETHOVEN; UMO-2018/31/G/ST1/03718)}}
\author[2]{James Davies \thanks{jgdavies@uwaterloo.ca}}
\author[1]{Bartosz Walczak \thanks{bartosz.walczak@uj.edu.pl; partially supported by the Polish National Science Centre grant 2019/34/E/ST6/00443}}
\affil[1]{Department of Theoretical Computer Science, Faculty of Mathematics and Computer Science, Jagiellonian~University, Kraków, Poland}
\affil[2]{Department of Combinatorics and Optimization, University of Waterloo, Waterloo, Canada}

\date{}

\begin{document}

\maketitle

\begin{abstract}
Extending the idea from the recent paper by Carbonero, Hompe, Moore, and Spirkl, for every function $f\colon\setN\to\setN\cup\set{\infty}$ with $f(1)=1$ and $f(n)\geq\binom{3n+1}{3}$, we construct a hereditary class of graphs $\calG$ such that the maximum chromatic number of a graph in $\calG$ with clique number $n$ is equal to $f(n)$ for every $n\in\setN$.
In particular, we prove that there exist hereditary classes of graphs that are $\chi$-bounded but not polynomially $\chi$-bounded.
\end{abstract}

\section{Introduction}

Given a class of graphs $\calC$ its \emph{$\chi$-bounding function} is the function $\chi_\calC\colon\setN\to\setN\cup\set{\infty}$ defined as
\[
    \chi_\calC(n)=\sup\set{\chi(G)\colon G\in\calC \text{ and } \omega(G)=n} \text{,}
\]
where $\chi(G)$ and $\omega(G)$ denote, respectively, the chromatic number and the clique number of $G$.
A class of graphs $\calC$ is \emph{$\chi$-bounded} if there is a function $f\colon\setN\to\setN$ such that $\chi(G)\leq f(\omega(G))$ for every graph $G\in\calC$, or equivalently if $\chi_\calC(n)$ is finite for every $n\in\setN$.
A class $\calC$ is \emph{polynomially\/ $\chi$-bounded} if such a function $f$ can be chosen to be a polynomial.
A class $\calC$ is \emph{hereditary} if it is closed under taking induced subgraphs.

A well-known and fundamental open problem, due to Esperet~\cite{esperet2017graph}, has been to decide whether every hereditary $\chi$-bounded class of graphs is polynomially $\chi$-bounded.
We provide a negative answer to this question.
More generally, we prove that $\chi$-bounding functions may be arbitrary, so long as they are bounded from below by a certain cubic function.

\begin{theorem}\label{thm:main}
Let\/ $f\colon\setN\to\setN\cup\set{\infty}$ be such that\/ $f(1)=1$ and\/ $f(n)\geq\binom{3n+1}{3}$ for every\/ $n\geq 2$.
Then there exists a hereditary class of graphs\/ $\calG$ such that\/ $\chi_{\calG}(n)=f(n)$ for every\/ $n\in\setN$.
\end{theorem}

On the other hand, $\chi$-bounding functions are not entirely arbitrary.
For instance, Scott and Seymour~\cite{SS2016} proved that every hereditary class of graphs $\calC$ with $\chi_\calC(2)=2$ satisfies $\chi_\calC(n)\leq 2^{2^{n+1}}$.

The proof of \cref{thm:main} is heavily based on the idea used by Carbonero, Hompe, Moore, and Spirkl~\cite{CHMS} in their very recent solution to another well-known problem attributed to Esperet~\cite{SS2020}.
They proved that for every $k\in\setN$, there is a $K_4$-free graph $G$ with $\chi(G)\geq k$ such that every triangle-free induced subgraph of $G$ has chromatic number at most $4$.
Their proof, in turn, relies on an idea by Kierstead and Trotter~\cite{KT1992}, who proved in 1992 that the class of oriented graphs excluding an directed path on four vertices as an induced subgraph is not $\chi$-bounded.
We derive \cref{thm:main} from the following result, which also generalises the recent breakthrough of Carbonero, Hompe, Moore, and Spirkl~\cite{CHMS}.

\begin{theorem}\label{col:main}
For every pair of integers\/ $n$ and\/ $k$ with\/ $k\geq n\geq 2$, there exists a graph\/ $G$ with clique number\/ $n$ and chromatic number\/ $k$ such that every induced subgraph of\/ $G$ with clique number\/ $m<n$ has chromatic number at most\/ $\binom{3m+1}{3}$.
\end{theorem}

In the case that $n$ is a prime number, we obtain a better bound.

\begin{theorem}\label{col:mainprime}
For every pair of integers\/ $p$ and\/ $k$ with $p$ a prime and\/ $k\geq p$, there exists a graph\/ $G$ with clique number\/ $p$ and chromatic number\/ $k$ such that every induced subgraph of\/ $G$ with clique number\/ $m<p$ has chromatic number at most\/ $\binom{m+2}{3}$.
\end{theorem}

In the first version of this paper~\cite{BDWv1} we only proved a weaker version of \cref{col:mainprime} with $\binom{m+2}{3}$ replaced by $m^{m^2}$.
Despite the worse bound obtained, this alternative proof may still be of independent interest.

Very recently, Girão, Illingworth, Powierski, Savery, Scott, Tamitegama, and Tan~\cite{GIPSSTT} independently proved another result generalising the fact that for every prime $p$, there are graphs with clique number $p$ and arbitrarily large chromatic number whose induced subgraphs with clique number at most $p-1$ have bounded chromatic number~\cite{BDWv1}.
They proved that for every graph $F$ with at least one edge, there are graphs of arbitrarily large chromatic number and the same clique number as $F$ in which every $F$-free induced subgraph has chromatic number at most some constant $c_F$ depending only on $F$.
They also showed the analogous statement where clique number is replaced by odd girth.
See~\cite{SS2020} and~\cite{SR2019} for recent surveys on $\chi$-boundedness and polynomial $\chi$-boundedness.

\section{Proof}

First, we show that \cref{col:main} implies \cref{thm:main}.

\begin{proof}[Proof of \cref{thm:main} assuming \cref{col:main}]
Fix a function $f\colon\setN\to\setN\cup\set{\infty}$ such that $f(1)=1$ and $f(n)\geq\binom{3n+1}{3}$ for every $n\geq 2$.
By \cref{col:main}, for every pair of integers $n$ and $k$ with $k\geq n\geq 2$, there exists a graph $H_{n,k}$ with clique number $n$ and chromatic number $k$ such that every induced subgraph of $H_{n,k}$ with clique number $m<n$ is $\binom{3m+1}{3}$-colourable.

We now consider two cases.
If $f(n)$ is finite, we put $\calH_n=\set{H_{n,f(n)}}$.
Otherwise $f(n)=\infty$, and we put $\calH_n=\set{H_{n,k}\colon k\geq n}$.
Finally, we let $\calH=\bigcup_{n=2}^\infty\calH_n$ and $\calG$ be the hereditary closure of $\calH$.

We now argue that $\chi_\calG(n)=f(n)$ for all $n\in\setN$.
The claim holds trivially for $n=1$, so assume $n\geq 2$.
If $f(n)=\infty$, then the sequence of graphs $\set{H_{n,k}\colon k\geq n}\subseteq\calG$ all have clique number equal to $n$ and have unbounded chromatic number, thus showing that $\chi_\calG(n)=\infty$, as claimed.
Otherwise, $f(n)$ is finite.
The graph $H_{n,f(n)}\in\calG$ shows that $\chi_\calG(n)\geq f(n)$.
For the reverse inequality, let $G\in\calG$ be such that $\omega(G)=n$.
Then there exist integers $k$ and $n^*$ with $k\geq n^*\geq n$ such that $G$ is an induced subgraph of $H_{n^*,k}\in\calH$.
The unique graph of $\calH$ with clique number $n$ is $H_{n,f(n)}$.
So if $n^*=n$, then $\chi(G)\leq\chi(H_{n,f(n)})=f(n)$, and if $n^*>n$, then $\chi(G)\leq\binom{3n+1}{3}$.
Combining these inequalities, we conclude that
\[
    f(n) \leq \chi_\calG(n) \leq \max\set*{\tbinom{3n+1}{3},\:f(n)} = f(n) \text{,}
\]
and the theorem follows.
\end{proof}

The rest of the paper is devoted to proving \cref{col:main}.
We begin with the following lemma.

\begin{lemma}\label{lem:basegraph}
For every positive integer\/ $k$, there is a graph\/ $G_k$ and an acyclic orientation of its edges with the following properties:
\begin{enumerate}
\item \label{lem:colours} $\chi(G_k)=k$;
\item \label{lem:unique} for every pair of vertices\/ $u$ and\/ $v$, there is at most one directed path from\/ $u$ to\/ $v$ in\/ $G_k$;
\item \label{lem:path} there is a directed path in\/ $G_k$ on\/ $k$ vertices;
\item \label{lem:colouring} there is a\/ $k$-colouring\/ $\phi$ of\/ $G_k$ such that\/ $\phi(u)\neq\phi(v)$ for any two distinct vertices\/ $u$ and\/ $v$ such that there is a directed path from\/ $u$ to\/ $v$ in\/ $G_k$.
\end{enumerate}
\end{lemma}

Various well-known constructions of triangle-free graphs with arbitrarily high chromatic number, such as Zykov's~\cite{Zykov1949} and Tutte's~\cite{Descartes1947,Descartes1954}, satisfy the condition of \cref{lem:basegraph} once the edges are oriented in a way that follows naturally from the construction.
See~\cite{CHMS} and~\cite{KT1992} for an explicit construction of the graphs $G_k$ with the appropriate acyclic orientations, based on Zykov's construction.
It is only implicit that the acyclic orientations of the graphs in~\cite{CHMS} and~\cite{KT1992} satisfy all of the properties in the conclusion of \cref{lem:basegraph}, so for the sake completeness we provide a proof based on Tutte's construction.

\begin{proof}[Proof of \cref{lem:basegraph}]
We proceed by induction on $k$.
The base case $k=1$ follows by taking a single-vertex graph as $G_1$.
For the induction step, assume $G_{k-1}$ is an acyclically oriented graph satisfying conditions \ref{lem:colours}--\ref{lem:colouring} for $k-1$.
To construct $G_k$, begin with a stable set $S$ with $\size{S}=(k-1)(\size{V(G_{k-1})}-1)+1$, and for every subset $X$ of $S$ with $\size{X}=\size{V(G_{k-1})}$, add an isomorphic copy $G_X$ of $G_{k-1}$ (with the same orientation as in $G_{k-1}$) and an arbitrary perfect matching between the vertices in $X$ and the vertices of $G_X$, oriented from $X$ to $G_X$.
This clearly preserves acyclicity of the orientation.
Since every vertex in $S$ has at most one edge to each copy $G_X$ of $G_{k-1}$, condition \ref{lem:unique} is preserved.
Any directed path on $k-1$ vertices in $G_X$ extends to a directed path on $k$ vertices in $G_k$ by adding a vertex from $S$, so \ref{lem:path} holds.
Any colouring of the copies $G_X$ of $G_{k-1}$ with a common palette of $k-1$ colours extends to a $k$-colouring of $G_k$ by using a single new colour on $S$, which shows that $\chi(G_k)\leq\chi(G_{k-1})+1$ and condition \ref{lem:colouring} is preserved.
Finally, suppose there exists a $(k-1)$-colouring of $G_k$.
Then, since $\size{S}>(k-1)(\size{V(G_{k-1})}-1)$, there is a monochromatic set $X\subseteq S$ with $\size{X}=\size{V(G_{k-1})}$.
The fact that $X$ and $G_X$ are connected by a perfect matching implies that at most $k-2$ colours are used on $G_X$, which contradicts the fact that $\chi(G_X)=\chi(G_{k-1})=k-1$.
Hence $\chi(G_k)=k$, as claimed in \ref{lem:colours}.
\end{proof}

For the rest of the argument, we fix an arbitrary sequence $(G_k)_{k\in\setN}$ of graphs given by \cref{lem:basegraph}.
Now, for every pair of positive integers $k$ and $p$, where $p$ is a prime number, we construct a graph $G_{k,p}$ by adding edges to $G_k$ as follows.

Let $\leq$ be the directed reachability order of the vertices of $G_k$, that is, $u\leq v$ if and only if there is a (unique) directed path from $u$ to $v$ in $G_k$.
Since the orientation of $G_k$ given by \cref{lem:basegraph} is acyclic, $\leq$ is indeed a partial order.
For every pair of vertices $u$ and $v$ in $G_k$ such that $u\leq v$, let $d(u,v)$ be the length of the unique directed path from $u$ to $v$ in $G_k$ (i.e., the number of edges in that path).
The graph $G_{k,p}$ has the same vertex set as $G_k$ and has the set $\{uv\colon u<v$ and $d(u,v)\not\equiv 0\pmod{p}\}$ as the edge set.
We consider each such edge $uv$ as oriented from $u$ to $v$.
Since the original (oriented) edges $uv$ of $G_k$ satisfy $u<v$ and $d(u,v)=1$, the graph $G_{k,p}$ contains $G_k$ as a subgraph.
Furthermore, every edge of $G_{k,p}$ connects vertices with different colours in a $k$-colouring $\phi$ of $G_k$ claimed in \cref{lem:basegraph}.
Therefore, $\chi(G_{k,p})=k$.
Furthermore $G_{k,p}$ is acyclic since $G_k$ is acyclic.

Next, we examine cliques in $G_{k,p}$ (and its induced subgraphs).
Since $G_{k,p}$ is acyclic, every clique of $G_{k,p}$ induces a transitive tournament.
Given a clique $C$ of an acyclic oriented graph, we let $t(C)$ be the first vertex of the transitive tournament induced by $C$.
We call $t(C)$ the \emph{tail} of $C$.
Given a clique $C$ of $G_{k,p}$, we let $r(C)$ be the subset of $\setZ_p$ such that $r(C)\equiv\set{d(t(C),v)\colon v\in C}\pmod{p}$.
We call $r(C)$ the \emph{residue} of the clique $C$.
Note that $0$ is always contained in $r(C)$ since $t(C)\in C$.
Furthermore $\size{C}=\size{r(C)}$, otherwise there would exist two distinct vertices $u,v\in C$ such that $d(t(C),u)\equiv d(t(C),v)\pmod{p}$, and so $d(u,v)\equiv 0\pmod{p}$, which would contradict the fact that $u$ and $v$ are adjacent.
This observation allows us to determine the clique number of $G_{k,p}$.

\begin{lemma}\label{lem:cliquebound}
For every positive integer\/ $k$ and every prime\/ $p\leq k$, the graph\/ $G_{k,p}$ has clique number\/ $p$.
\end{lemma}

\begin{proof}
Since $G_k$ contains a directed path on $k$ vertices and $p\leq k$, the graph $G_{k,p}$ contains a clique of size $p$.
Conversely, if $C$ is a clique in $G_{k,p}$, then $\size{C}=\size{r(C)}\leq\size{\setZ_p}=p$.
\end{proof}

A \emph{rotation} of a subset $X$ of $\setZ_p$ is a subset of $\setZ_p$ of the form $X+a=\set{x+a\colon x\in X}$ for any $a\in\setZ_p$.
A subset of $\setZ_p$ is \emph{rooted} if it contains $0$.
The rotation $X+a$ of a rooted subset $X$ of $\setZ_p$ is rooted if and only if $-a\in X$.
Let $\sim_p$ be the equivalence relation on the rooted subsets of $\setZ_p$ such that $X\sim_pY$ whenever $Y$ is a rotation of $X$.
Let $[X]_p$ denote the equivalence class of $X$ in $\sim_p$.
For every proper rooted subset $X$ of $\setZ_p$ (such that $X\neq\setZ_p$), since $p$ is a prime, all rotations $X+a$ of $X$ with $a\in\setZ_p$ are distinct, and in particular $\size{[X]_p}=\size{X}$.
Order every equivalence class arbitrarily, and for every proper rooted subset $X$ of $\setZ_p$, let $c(X)\in\set{1,\ldots,\size{X}}$ denote the position of $X$ in this ordering.

\begin{lemma}\label{lem:chibound}
For every positive integer\/ $k$, every prime\/ $p$, and every induced subgraph\/ $G$ of\/ $G_{k,p}$ with clique number\/ $m<p$, we have\/ $\chi(G)\leq\binom{m+2}{3}$.
\end{lemma}

\begin{proof}
We will colour the vertices of $G$ by triples of integers $(a,b,c)$ with $m\geq a\geq b\geq c\geq 1$.
Since there are $\binom{m+2}{3}$ choices for such a triple, this will be a $\binom{m+2}{3}$-colouring of $G$.

For each vertex $v$ of $G$, let $a(v)$ be the maximum size of a clique in $G$ with tail $v$.
Thus $m\geq a(v)\geq 1$.
Let $B(v)$ be the intersection of the residues of all cliques of size $a(v)$ with tail $v$ in $G$.
Since $0$ belongs to the residue of every clique, we have $0\in B(v)$.
Let $b(v)=\size{B(v)}$, so that $a(v)\geq b(v)\geq 1$.
Let $c(v)=c(B(v))$, so that $b(v)\geq c(v)\geq 1$, as $\size{[B(v)]_p}=\size{B(v)}=b(v)$.
Finally, let $\psi(v)=(a(v),b(v),c(v))$.
We have $m\geq a(v)\geq b(v)\geq c(v)\geq 1$ for every $v$, so it remains to show that $\psi$ is a proper colouring of $G$.

Suppose for the sake of contradiction that some two vertices $u$ and $v$ of $G$ with $\psi(u)=\psi(v)$ are connected by an edge of $G$ oriented from $u$ to $v$.
Let $d\in\setZ_p$ be such that $d(u,v)\equiv d\pmod{p}$.
Since $u$ and $v$ are adjacent in $G$, we have $d\neq 0$.
Observe that if $X$ is the residue of a clique with tail $v$, then $(X+d)\cup\set{0}$ is the residue of a clique with tail $u$.
Therefore, since $a(u)=a(v)$, the residue of every clique of size $a(v)$ with tail $v$ must contain $-d$.
Thus $-d\in B(v)$ and if $X$ is the residue of a clique of size $a(v)$ with tail $v$, then $X+d$ is the residue of a clique of the same size with tail $u$.
Hence $B(u)\subseteq B(v)+d$, and since $b(u)=b(v)$, we further conclude that $B(u)=B(v)+d$.
Since $0$ belongs to the residue of every clique, both $B(u)$ and $B(v)$ are rooted and $B(u)\sim_pB(v)$.
Thus $B(u)=B(v)$, as $c(u)=c(v)$.
However, since $b(u)=b(v)\leq m<p$ and $d\neq 0$, we have $B(u)=B(v)+d\neq B(v)$, which is a constradiction.
This shows that $\psi$ is a proper colouring of $G$, as desired.
\end{proof}

By combining \cref{lem:cliquebound} and \cref{lem:chibound}, we have so far proven \cref{col:mainprime}.
Next we extend the construction to non-primes in order to prove \cref{col:main}.

For every triple of positive integers $k,n,p$ with $p$ prime and $p>2n$, we construct a graph $G_{k,n,p}$ by removing the edges $uv$ of $G_{k,p}$ with $d(u,v) \not\equiv\pm 1,\ldots,\pm(n-1)\pmod p$.
We will now determine the clique number and the chromatic number of $G_{k,n,p}$.

\begin{lemma}\label{lem:cliquenonprime}
Let\/ $k$, $n$, and\/ $p$ be positive integers with\/ $p$ prime, $p>2n$, and\/ $k\geq n$.
Then\/ $G_{k,n,p}$ has clique number\/ $n$ and chromatic number\/ $k$.
\end{lemma}

\begin{proof}
We have $\chi(G_k)=\chi(G_{k,p})=k$.
Since $G_k$ is a subgraph of $G_{k,n,p}$ and $G_{k,n,p}$ is a subgraph of $G_{k,p}$, it follows that $\chi(G_{k,n,p})=k$.
Next, we determine the clique number of $G_{k,n,p}$.

Let $I=\set{n,n+1,\ldots,p-n}\subset\setZ_p$.
Note that $uv$ is an edge of $G_{k,n,p}$ if and only if $u<v$ and $d(u,v)\notin\set{0}\cup I\pmod{p}$.
Since $G_k$ contains a directed path on $k$ vertices and $n\leq k$, the graph $G_{k,n,p}$ has clique number at least $n$.
It remains to show that $G_{k,n,p}$ has clique number at most $n$.

Let $C$ be a clique in $G_{k,n,p}$, and let $v=t(C)$.
Consider a vertex $x\in C$.
Since $-I=I$ (with arithmetic modulo $p$), we can observe that $r(C)$ must be disjoint from $I+d(v,x)$.
Indeed, if there is a vertex $y\in C$ with $d(v,y)\in I+d(v,x)$, then either $y<x$ and $d(y,x)=d(v,x)-d(v,y)=-(d(v,y)-d(v,x))\in-I=I\pmod{p}$, or $x<y$ and $d(x,y)=d(v,y)-d(v,x)\in I\pmod{p}$, and in either case $xy$ is not an edge of $G_{k,m,p}$, contradicting the fact that $C$ is a clique.
Thus the set $r(C)$ is disjoint from $\bigcup_{x\in C}(I+d(v,x))$, which implies $\size{r(C)}+\size{\bigcup_{x\in C}(I+d(v,x))}\leq p$.
For every $y\in r(C)\cap\set{1,\ldots,n-1}$, we have $p-n+y\in\bigcup_{x\in C}(I+d(v,x))$, and similarly, for every $y\in r(C)\cap\set{-1,\ldots,-n+1}$, we have $n+y\in\bigcup_{x\in C}(I+d(v,x))$.
Therefore, since $d(v,x)\not\equiv d(v,y)\pmod{p}$ for any distinct $x,y\in C$, we have
\[
    \textstyle\size*{\bigcup_{x\in C}(I+d(v,x))} \geq \size{I} + \size{r(C)\setminus\set{0}} = \size{r(C)}+p-2n \text{.}
\]

Finally using the fact that $\size{r(C)}=\size{C}$, we conclude that $\size{C}\leq n$, as desired.
\end{proof}

Next we examine the maximum size of a clique in an induced subgraph of $G_{k,n,p}$ that is induced by the vertices of a clique in $G_{k,p}$.
This will allow us to compare the chromatic number of induced subgraphs of $G_{k,p}$ and $G_{k,n,p}$ that have the same vertex set.

\begin{lemma}\label{lem:inducedcliquenonprime}
Let\/ $k$, $n$, and\/ $p$ be positive integers with\/ $p$ prime, $p>2n$, and\/ $k\geq n$.
Then for every clique\/ $C$ of\/ $G_{k,p}$, the induced subgraph\/ $G_{k,n,p}[C]$ of\/ $G_{k,n,p}$ contains a clique of size at least\/ $\frac{n}{p}\size{C}$.
\end{lemma}

\begin{proof}
Let $C$ be a clique in $G_{k,p}$.
For each $i\in\setZ_p$, let $J_i=\set{i,i+1,\ldots,i+n-1}\subset\setZ_p$.
Since each $i\in\setZ_p$ is contained in exactly $n$ of the $p$ sets $J_0,\ldots,J_{p-1}$, by the pigeon-hole principle, there exists $i\in\setZ_p$ such that $\size{r(C)\cap J_i}\geq\frac{n}{p}\size{r(C)}$.
Let $C_i=\set{v\in C\colon d(t(C),v)\in J_i\pmod{p}}$.
It follows that $\size{C_i}=\size{r(C)\cap J_i}\geq\frac{n}{p}\size{r(C)}=\frac{n}{p}\size{C}$.
It remains to show that $C_i$ is a clique in $G_{k,n,p}$.

Let $x$ and $y$ be distinct vertices in $C_i$.
Since $x,y\in C$, they are adjacent in $G_{k,p}$, so $d(t(C),x)\not\equiv d(t(C),y)\pmod{p}$, and we can assume without loss of generality that $x<y$.
It follows that $d(x,y)=d(t(C),y)-d(t(C),x)\in\set{\pm 1,\ldots,\pm(n-1)}\pmod{p}$, as $d(t(C),x),\:d(t(C),y)\in J_i\pmod{p}$.
Hence $x$ and $y$ are adjacent in $G_{k,n,p}$.
We conclude that $C_i$ is indeed a clique in $G_{k,n,p}$.
\end{proof}

\begin{lemma}\label{lem:chiboundonprime}
Let\/ $k$, $n$, and\/ $p$ be positive integers with\/ $p$ prime, $p>2n$, and\/ $k\geq n$, and let\/ $G$ be an induced subgraph of\/ $G_{k,n,p}$ with\/ $m=\omega(G)<n$.
Then\/ $\chi(G)\leq\binom{\floor{mp/n}+2}{3}$.
\end{lemma}

\begin{proof}
Let $G'=G_{k,p}[V(G)]$.
\cref{lem:inducedcliquenonprime} yields $\omega(G')\leq\floor{mp/n}$.
The fact that $G$ is a subgraph of $G'$ and \cref{lem:chibound} yield $\chi(G)\leq\chi(G')\leq\binom{\floor{mp/n}+2}{3}$.
\end{proof}

\cref{col:main} now follows from \cref{lem:cliquenonprime}, \cref{lem:chiboundonprime}, and the following theorem of Schur~\cite{Schur1929} on the gaps between prime numbers.

\begin{theorem}\label{thm:primes}
For every integer\/ $n\geq 2$, there is a prime\/ $p$ such that\/ $2n<p<3n$.
\end{theorem}

\section{Concluding remarks}

To better understand $\chi$-bounding functions, it is of course of interest to improve the bound of $\binom{3m+1}{3}$ in \cref{col:main} (and equivalently this same lower bound function for $f$ in \cref{thm:main}).

A slight tweak to the last step of the proof improves this bound slightly to $\binom{2m}{3}+o(m^3)$.
To do this, instead of using \cref{thm:primes}, we can use the fact that for any $\epsilon>0$, there exists a $n_\epsilon$ such that for every $n\geq n_\epsilon$, there is always a prime $p$ with $2n<p<(2+\epsilon)n$.
This follows from the prime number theorem that the number of primes at most $n$ is asymptotically equal to $n/\ln n$.
For a more recent and explicit result on the gaps between primes, see~\cite{Dusart2018}.

One may hope that another way to further improve this bound would be to improve the bound of $\binom{m+2}{3}$ in \cref{lem:chibound}.
However, in our construction, \cref{lem:chibound} is in some sense best possible.
For every prime $p$, we have been able to construct a graph $G_k'$ (with $k$ large enough) that satisfies the conclusion of \cref{lem:basegraph}, and such that for every positive integer $m<p$, the graph $G_{k,p}'$ (as constructed from $G_k'$) contains an induced subgraph with clique number $m$ and chromatic number $\binom{m+2}{3}$.
So any improvements would require an entirely new construction.

In the other direction, the only result restricting $\chi$-bounding functions is that of Scott and Seymour~\cite{SS2016} stating that if a hereditary class of graphs $\calC$ satisfies $\chi_\calC(2)\leq 2$, then $\calC$ is $\chi$-bounded.
We conjecture the following generalisation.

\begin{conjecture}
For every integer\/ $k\geq 2$, if\/ $\calC$ is a hereditary class of graphs such that\/ $\chi_\calC(n)\leq k$ for every positive integer\/ $n\leq k$, then the class\/ $\calC$ is\/ $\chi$-bounded.
\end{conjecture}

\end{document}